\documentclass[11pt]{article}
\usepackage{amsmath,amsthm,amssymb}
\usepackage[dvips]{graphicx}

\newcommand{\C}{\mathbb C}
\newcommand{\HH}{{\cal H}}
\newcommand{\PP}{{\mathbb P}}

\newcommand{\RR}{{\cal R}}
\newcommand{\T}{{\cal T}}
\newcommand{\U}{{\cal U}}

\newcommand{\Tn}{\T_n^{-1}([-1,1])}

\newcommand{\re}{\operatorname{Re}}

\newcommand{\e}{{\operatorname{e}}}
\newcommand{\ii}{{\operatorname{i}}}
\newcommand{\dd}{{\operatorname{d}}}
\newcommand{\CAP}{\operatorname{cap}}
\newcommand{\ov}{\overline}

\begin{document}


\title{The P\'olya-Chebotarev Problem and Inverse Polynomial Images\footnote{published in: Acta Mathematica Hungarica {\bf ??} (2013), ??--??.}}
\author{Klaus Schiefermayr\footnote{University of Applied Sciences Upper Austria, School of Engineering and Environmental Sciences, Stelzhamerstr.\,23, 4600 Wels, Austria, \textsc{klaus.schiefermayr@fh-wels.at}}}
\maketitle

\theoremstyle{plain}
\newtheorem{theorem}{Theorem}
\newtheorem{conjecture}{Conjecture}
\newtheorem{lemma}{Lemma}
\theoremstyle{definition}
\newtheorem*{remark}{Remark}
\newtheorem*{example}{Example}

\begin{abstract}
Consider the problem, usually called the P\'olya-Chebotarev problem, of finding a continuum in the complex plane including some given points such that the logarithmic capacity of this continuum is minimal. We prove that each connected inverse image $\T_n^{-1}([-1,1])$ of a polynomial $\T_n$ is always the solution of a certain P\'olya-Chebotarev problem. By solving a nonlinear system of equations for the zeros of $\T_n^2-1$, we are able to construct polynomials $\T_n$ with a connected inverse image.
\end{abstract}



\section{The P\'olya-Chebotarev Problem}


Let us introduce an extremal problem concerning the logarithmic capacity usually called the ``P\'olya-Chebotarev problem'' or the ``Chebotarev problem'', since Chebotarev mentioned the problem in a letter to P\'olya\,\cite{Polya-1929}.
\begin{enumerate}
\item[] \emph{P\'olya-Chebotarev problem}: Given $\nu$ distinct points $c_1,c_2,\ldots,c_{\nu}\in\C$ in the complex plane, find a \emph{continuum} $S$, i.e.\ $S$ is \emph{connected}, with the property that $c_1,\ldots,c_{\nu}\in{S}$, such that the logarithmic capacity $\CAP{S}$ is minimal.
\end{enumerate}
Existence and uniqueness of such a set $S$ is proved in \cite{Groetzsch-1930}. We will call the minimal set $S$ the \emph{P\'olya-Chebotarev continuum} for $\{c_1,c_2,\ldots,c_{\nu}\}$.

The case of $\nu=2$ points $c_1,c_2$ is trivial: in this situation the P\'olya-Chebotarev continuum is just the segment $[c_1,c_2]$.

In \cite{OrtegaPridhnani-2010}, the authors gave a short history of the P\'olya-Chebotarev problem and a description of the solution which can be implemented numerically. Fedorov~\cite{Fedorov-1985} gave a complete solution of the P\'olya-Chebotarev problem for $\nu=3$ and $\nu=4$ points (only a symmetric case) with the help of the Jacobian elliptic functions, see also \cite{Sch-2013a}.


In the following, let us recall the well-known representation of the P\'olya-Chebotarev continuum with the help of a hyperelliptic integral, see, e.g., \cite{GrassmannRokne-1975} or \cite{GrassmannRokne-1978}.


\begin{theorem}\label{T-Chebotarev}
Let $c_1,c_2,\ldots,c_{\nu}\in\C$ be $\nu\geq3$ given pairwise distinct complex points. If there exists $\nu-2$ points $d_1,\ldots,d_{\nu-2}\in\C$ with
\begin{equation}
\begin{aligned}
\re\bigl\{\Phi(c_j)\bigr\}&=0,\qquad{j}=1,2,\ldots,\nu,\\
\re\bigl\{\Phi(d_j)\bigr\}&=0,\qquad{j}=1,2,\ldots,\nu-2,
\end{aligned}
\end{equation}
where $\Phi(z)$ is the hyperelliptic integral
\begin{equation}\label{Phi}
\Phi(z)=\int_{c_1}^{z}\frac{\sqrt{\prod_{j=1}^{\nu-2}(w-d_j)}}{\sqrt{\prod_{j=1}^{\nu}(w-c_j)}}\,\dd{w},
\end{equation}
then the set $S$ defined by
\begin{equation}
S:=\bigl\{z\in\C:\re\{\Phi(z)\}=0\bigr\},
\end{equation}
is the P\'olya-Chebotarev continuum for $\{c_1,c_2,\ldots,c_{\nu}\}$.
\end{theorem}


\begin{remark}
\begin{enumerate}
\item Let us mention that the function $\re\{\Phi(z)\}$ is harmonic and single valued in $\ov\C\setminus{S}$ and is the unique Green function with pole at $\infty$ for $\C\setminus{S}$, see \cite{NuttallSingh-1977}. For the definition and many properties of the Green function and the logarithmic capacity of a set $S$ in the complex plane, we refer to \cite{Ransford-1995} and \cite{Ransford-2010}.
\item We will call the complex points $d_1,\ldots,d_{\nu-2}$, which appear in formula \eqref{Phi}, the \emph{bifurcation points} of the P\'olya-Chebotarev continuum $S$. If a point $d^*\in\C$ is a zero of multiplicity $k$ of the polynomial $\prod_{j=1}^{\nu-2}(w-d_j)$ then $d^*$ is called a bifurcation point of multiplicity $k$. For a typical P\'olya-Chebotarev continuum for $\nu=5$ points, see Fig.\,\ref{Fig_ChebotarevSet-0}.
\end{enumerate}
\end{remark}


\begin{figure}[ht]
\begin{center}
\includegraphics[scale=0.8]{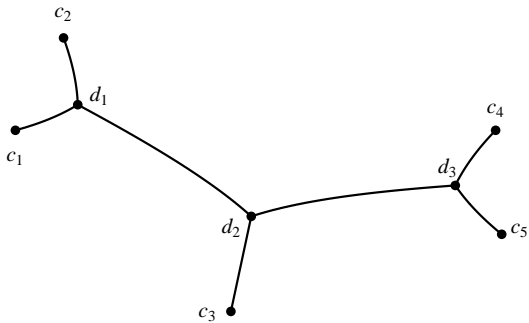}
\caption{\label{Fig_ChebotarevSet-0} Illustration of a P\'olya-Chebotarev continuum ($\nu=5$)}
\end{center}
\end{figure}


The paper is organized as follows. In Section\,2, after two preliminary lemmata, the main result is stated in Theorem\,2. We prove that each inverse image $\T_n^{-1}([-1,1])$ of a polynomial $\T_n$, which is connected, solves a certain P\'olya-Chebotarev problem. Further, we give some simple examples of polynomials with a connected inverse image. In Section\,3, we show how to construct polynomials $\T_n$ with a connected inverse image by solving a certain nonlinear system of equations for the zeros of $\T_n^2-1$. Some examples and a conjecture concerning a denseness statement conclude the paper.

\section{The Connection between the P\'olya-Chebotarev Problem and Inverse Polynomial Images}


Let us start with a lemma which is very important for what follows. The first part, see \cite[Lemma\,2]{Sch-2012a}, is just an immediate consequence of the fundamental theorem of algebra, the second part is due to Peherstorfer\,\cite[Corollary\,2]{Peh-1996}. As usual, let $\PP_n$ denote the set of all polynomials of degree $n$ with complex coefficients.


\begin{lemma}\label{L-THU}
For any polynomial $\T_n(z)=\tau{z}^n+\ldots\in\PP_n$, $\tau\in\C\setminus\{0\}$, there exists a unique $\ell\in\{1,2,\ldots,n\}$, a unique monic polynomial
\begin{equation}\label{H}
\HH_{2\ell}(z)=\prod_{j=1}^{2\ell}(z-a_j)=z^{2\ell}+\ldots\in\PP_{2\ell}
\end{equation}
with pairwise distinct zeros $a_1,a_2,\ldots,a_{2\ell}$, and a unique polynomial\\
$\U_{n-\ell}(z)=\tau{z}^{n-\ell}+\ldots\in\PP_{n-\ell}$ such that the polynomial equation
\begin{equation}\label{THU}
\T_n^2(z)-1=\HH_{2\ell}(z)\,\U_{n-\ell}^2(z)
\end{equation}
holds.\\
\indent Further, there exists a monic polynomial $\RR_{\ell-1}(z)=z^{\ell-1}+\ldots\in\PP_{\ell-1}$ such that
\begin{equation}\label{R}
\T_n'(z)=n\,\RR_{\ell-1}(z)\,\U_{n-\ell}(z)
\end{equation}
and, for $z\in\C$ with $\T_n(z)\notin[-1,1]$,
\begin{equation}
\T_n(z)=\pm\cosh\Bigl(n\int_{a_j}^{z}\frac{\RR_{\ell-1}(w)}{\sqrt{\HH_{2\ell}(w)}}\,\dd{w}\Bigr),
\end{equation}
where $a_j$ is any zero of $\HH_{2\ell}$.
\end{lemma}


Note that the points $a_1,a_2,\ldots,a_{2\ell}$ are exactly those zeros of $\T_n^2(z)-1$ which have odd multiplicity.


Next, let us introduce the notion of the inverse image of a polynomial. For a polynomial $\T_n\in\PP_n$, let $\Tn$ denote the inverse image of $[-1,1]$ with respect to $\T_n$, i.e.
\begin{equation}
\Tn:=\bigl\{z\in\C:\T_n(z)\in[-1,1]\bigr\}.
\end{equation}
In general, the inverse image $\Tn$ consists of $n$ Jordan arcs \cite{Sch-2012a}. One reason why we are interested in the factorization of $\T_n^2(z)-1$ in the form \eqref{THU} is that the number $\ell$ signifies the \emph{minimum number of Jordan arcs} the inverse image $\Tn$ consists of, see \cite[Theorem\,2]{Sch-2012a}.


Concerning the connectivity of the inverse image $\Tn$, the following is known.


\begin{lemma}\label{L-Connectedness}
Let $\T_n\in\PP_n$, and let $\RR_{\ell-1}$ as in Lemma\,\ref{L-THU}. The following statements are equivalent:
\begin{enumerate}
\item The inverse image $\T_n^{-1}([-1,1])$ is connected.
\item All zeros of the derivative $\T_n'$ are located in $\T_n^{-1}([-1,1])$.
\item All zeros of $\RR_{\ell-1}$ are located in $\T_n^{-1}([-1,1])$.
\end{enumerate}
\end{lemma}
\begin{proof}
The equivalence of (i) and (ii) is proved in \cite[Theorem\,4]{Sch-2012a}. The equivalence of (ii) and (iii) follows immediately from \eqref{R}, since, by \eqref{THU}, all zeros of $\U_{n-\ell}$ are located in $\Tn$.
\end{proof}


Next, we state and prove the first main result. It says that each polynomial $\T_n\in\PP_n$ with a connected inverse image $\Tn$ solves a certain P\'olya-Chebotarev problem.


\begin{theorem}\label{T-InvPolImage}
Let $\T_n(z)=\tau{z}^n+\ldots\in\PP_n$ with $\tau\in\C\setminus\{0\}$ and suppose that all zeros of the derivative $\T_n'$ are located in the inverse image $\Tn$, i.e.\ $\Tn$ is connected. Suppose that $c_1,c_2,\ldots,c_{\nu}$ (pairwise distinct) are exactly the simple zeros of $\T_n^2(z)-1$. Then $S=\Tn$ is the P\'olya-Chebotarev continuum for $\{c_1,c_2,\ldots,c_{\nu}\}$ with corresponding minimal logarithmic capacity
\begin{equation}\label{cap}
\CAP{S}=\frac{1}{\sqrt[n]{2|\tau|}}.
\end{equation}
\end{theorem}


\begin{proof}
For $\T_n\in\PP_n$, let $\HH_{2\ell}(z)$, $\U_{n-\ell}(z)$, and $\RR_{\ell-1}(z)$ as in Lemma\,\ref{L-THU}. By assumption and Lemma\,\ref{L-THU}, $c_1,\ldots,c_{\nu}\in\{a_1,\ldots,a_{2\ell}\}$ are those zeros of $\HH_{2\ell}(z)$, which are not zeros of $\U_{n-\ell}(z)$. Let $b_1,\ldots,b_{2\ell-\nu}\in\{a_1,\ldots,a_{2\ell}\}$ be those zeros of $\HH_{2\ell}(z)$, which are also zeros of $\U_{n-\ell}(z)$. Then, by \eqref{THU}, each $b_j$ is a zero of $\T_n^2(z)-1$ with an odd multiplicity greater than two. Thus, by \eqref{THU} and \eqref{R}, each $b_j$ is a zero of $\RR_{\ell-1}(z)$. Hence
\[
\frac{\RR_{\ell-1}(z)}{\sqrt{\HH_{2\ell}(z)}}=\frac{\sqrt{\RR_{\ell-1}^2(z)}}{\sqrt{\prod_{j=1}^{\nu}(z-c_j)\cdot\prod_{j=1}^{2\ell-\nu}(z-b_j)}}
=\frac{\sqrt{\prod_{j=1}^{\nu-2}(z-d_j)}}{\sqrt{\prod_{j=1}^{\nu}(z-c_j)}},
\]
where
\begin{equation}\label{dj}
\prod_{j=1}^{\nu-2}(z-d_j):=\frac{\RR_{\ell-1}^2(z)}{\prod_{j=1}^{2\ell-\nu}(z-b_j)}.
\end{equation}
Define
\[
\Phi(z):=\int_{c_1}^{z}\frac{\RR_{\ell-1}(w)}{\sqrt{\HH_{2\ell}(w)}}\,\dd{w},
\]
then, by Lemma\,\ref{L-THU},
\begin{gather*}
\bigl\{z\in\C:\re\Phi(z)=0\bigr\}=\bigl\{z\in\C:\cosh(n\Phi(z))\in[-1,1]\bigr\}\\
=\bigl\{z\in\C:\T_n(z)\in[-1,1]\bigr\}=\T_n^{-1}([-1,1]).
\end{gather*}
By construction, $c_1,\ldots,c_{\nu}\in\Tn$, thus $\re\Phi(c_j)=0$, $j=1,\ldots,\nu$. By \eqref{dj}, $d_1,\ldots,d_{\nu-2}$ are zeros of $\RR_{\ell-1}$, thus, by Lemma\,\ref{L-Connectedness}, $d_1,\ldots,d_{\nu-2}\in\Tn$ and again $\re\Phi(d_j)=0$, $j=1,\ldots,\nu-2$. Now the assertion apart from formula \eqref{cap} follows by Theorem\,\ref{T-Chebotarev} and it remains to prove \eqref{cap}.\\
\indent For any polynomial $\T_n(z)=\tau{z}^n+\ldots\in\PP_n$, $\tau\in\C\setminus\{0\}$, it is known, see \cite{Peh-1996}, \cite{KamoBorodin} or \cite{OPZ}, that the monic polynomial $\hat{\T}_n(z):=\T_n(z)/\tau=z^n+\ldots\in\PP_n$ is the Chebyshev polynomial on its inverse image $S:=\T_n^{-1}([-1,1])$ with minimum deviation $L_n(S)=1/|\tau|$. Thus, from the identity $L_n(S)=2(\CAP{S})^n$, see \cite[Theorem\,1]{Sch-2008}, we get \eqref{cap}.
\end{proof}


\begin{remark}
\begin{enumerate}
\item Let $\T_n\in\PP_n$ and $S:=\Tn$. Then~\cite{Peh-1996}
\begin{equation}
g_S(z)=\tfrac{1}{n}\log\left|\T_n(z)+\sqrt{\T_n^2(z)-1}\right|
\end{equation}
is the Green function of $\ov{\C}\setminus{S}$ with pole at infinity. For the complex Green function
\begin{equation}
G_S(z)=\exp\left(-g_S(z)+\ii{h}_S(z)\right),
\end{equation}
where $h_S(z)$ is the harmonic conjugate of $g_S(z)$, we have~\cite{Peh-1996}
\begin{equation}
G_S(z)=\exp\left(-\int_{c_1}^{z}\frac{\RR_{\ell-1}(w)}{\sqrt{\HH_{2\ell}(w)}}\,\dd{w}\right),
\end{equation}
where $\RR_{\ell-1}$ and $\HH_{2\ell}$ are as in Lemma\,\ref{L-THU}.
\item In \cite{Sch-2013a}, the author characterizes the P\'olya-Chebotarev continuum with the help of Jacobian elliptic and theta functions (using Zolotarev's conformal mapping) for the case of three and four points (i.e.\ $\nu=3$ and $\nu=4$). With this characterization, one can construct sets $\{c_1,c_2,c_3,c_4\}$ such that the corresponding P\'olya-Chebotarev continuum is \emph{not} an inverse polynomial image.
\end{enumerate}
\end{remark}


Let us give some examples of polynomials with a connected inverse image.


\begin{example}
\begin{enumerate}
\item Let $\T_n(z):=z^n$. The inverse image of $\T_n$ is
\begin{equation}\label{S-1}
S:=\Tn=\bigl\{z=r\e^{\ii\varphi}:r\in[0,1],\varphi\in\{0,\tfrac{\pi}{n},\tfrac{2\pi}{n},\ldots,\tfrac{(2n-1)\pi}{n}\}\bigr\},
\end{equation}
see Fig.\,\ref{Fig_ChebotarevSet-1}, where $n=5$. By Theorem\,\ref{T-InvPolImage}, the set $S$ in \eqref{S-1} is the P\'olya-Chebotarev continuum for the $\nu=2n$ points $\{c_0,c_1,\ldots,c_{2n-1}\}$, where $c_k:=\e^{\ii{k}\pi/n}$, $k=0,1,\ldots,2n-1$, with (the only) bifurcation point $d_1=0$ (of multiplicity $2n-2$). With the notations of Lemma\,\ref{L-THU}, we have $\ell=n$,
\[
\HH_{2\ell}(z)=\prod_{k=0}^{2n-1}(z-\e^{\ii{k}\pi/n})=z^{2n}-1, \quad
\U_{n-\ell}(z)=1 \quad\text{and}\quad \RR_{\ell-1}(z)=z^{n-1}.
\]
Although very simple, this example seems to be new since we could not find it in the literature.
\begin{figure}[ht]
\begin{center}
\includegraphics[scale=0.8]{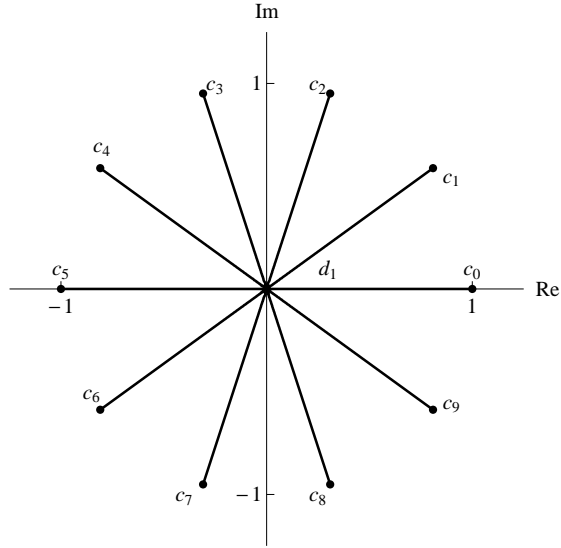}
\caption{\label{Fig_ChebotarevSet-1} The inverse image of $[-1,1]$ with respect to the polynomial $\T_n(z)=z^n$, $n=5$}
\end{center}
\end{figure}
\item Let $n=2$ and $\T_n(z):=\frac{2}{1+\alpha^2}(z^2-1)+1$ with $0\leq\alpha<\infty$. The inverse image of $\T_n$ is the cross
\begin{equation}\label{S-2}
S:=\Tn=[-1,1]\cup[-\ii\alpha,\ii\alpha].
\end{equation}
By Theorem\,\ref{T-InvPolImage}, the set $S$ in \eqref{S-2} is the P\'olya-Chebotarev continuum for the $\nu=4$ points $\{c_1,c_2,c_3,c_4\}=\{-1,1,-\ii\alpha,\ii\alpha\}$ with (the only) bifurcation point $d_1=0$ (of multiplicity two), see Fig.\,\ref{Fig_ChebotarevSet-2}. With the notations of Lemma\,\ref{L-THU}, we have $\ell=2$,
\[
\HH_{2\ell}(z)=(z^2-1)(z^2+\alpha^2), \quad
\U_{n-\ell}(z)=\frac{2}{1+\alpha^2} \quad\text{and}\quad \RR_{\ell-1}(z)=z.
\]
\begin{figure}[ht]
\begin{center}
\includegraphics[scale=0.9]{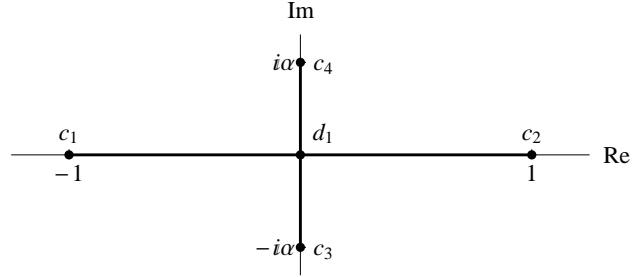}
\caption{\label{Fig_ChebotarevSet-2} P\'olya-Chebotarev continuum for the four points $\{-1,1,-\ii\alpha,\ii\alpha\}$}
\end{center}
\end{figure}
\item Let $n=3$, $0\leq\alpha<\infty$, and define
\begin{equation}\label{T3}
\T_n(z):=-1-\frac{(z-1)(2z+1-\alpha^2)^2}{(1+\alpha^2)^2}.
\end{equation}
The zeros of $\T_n^2(z)-1$ are $\pm1$, $\tfrac{1}{2}(1+\alpha^2)\pm\ii\alpha$ (multiplicity one) and $\tfrac{1}{2}(\alpha^2-1)$ (multiplicity two). The zeros of $\T_n'(z)$ are $\tfrac{1}{2}(\alpha^2-1)$ and $\tfrac{1}{6}(3+\alpha^2)$. For $\alpha\in[0,\sqrt{3}]$, the inverse image $\Tn$ is connected and consists of the interval $[-1,1]$ and a hyperbola moving from $\tfrac{1}{2}(1+\alpha^2)+\ii\alpha$ to $\tfrac{1}{2}(1+\alpha^2)-\ii\alpha$ crossing the interval $[-1,1]$ at $\tfrac{1}{6}(3+\alpha^2)$. By Theorem\,\ref{T-InvPolImage}, the set $S=\Tn$ is the P\'olya-Chebotarev continuum for the $\nu=4$ points
\[
\{c_1,c_2,c_3,c_4\}=\{-1,1,\tfrac{1}{2}(1+\alpha^2)-\ii\alpha,\tfrac{1}{2}(1+\alpha^2)+\ii\alpha\}
\]
with the bifurcation point $d_1=\tfrac{1}{6}(3+\alpha^2)$ (of multiplicity two). In Fig.\,\ref{Fig_ChebotarevSet-3}, we have plotted the inverse image $\Tn$ of the polynomial $\T_n(z)$ defined in \eqref{T3} for $\alpha=\tfrac{1}{2}$. For $\alpha=0$, we get $\T_n(z)=3z-4z^3$, i.e.\ the Chebyshev polynomial of the first kind of degree three, for $\alpha=\sqrt{3}$, we get $\T_n(z)=-1-\tfrac{1}{4}(z-1)^3$.
\begin{figure}[ht]
\begin{center}
\includegraphics[scale=0.8]{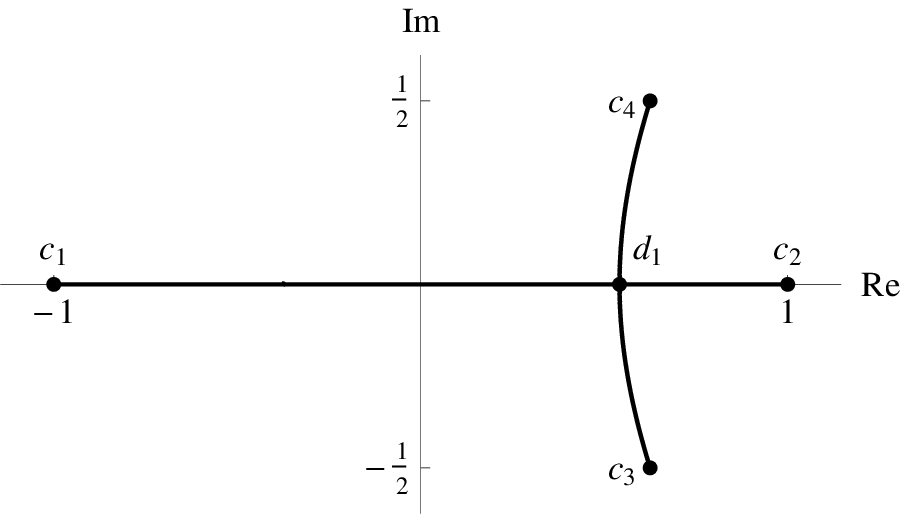}
\caption{\label{Fig_ChebotarevSet-3} The inverse image $\Tn$ of the polynomial $\T_n(z)$ defined in \eqref{T3} for $\alpha=\tfrac{1}{2}$}
\end{center}
\end{figure}
\item Let $n=4$, $0\leq\alpha<\infty$, and define
\begin{equation}\label{T4}
\T_n(z):=\frac{8z^4-8z^2+1+4\alpha^2}{1+4\alpha^2}.
\end{equation}
The zeros of $\T_n^2(z)-1$ are $\pm1$, $\pm\frac{1}{2}\beta\pm\ii\alpha/\beta$ (multiplicity one) and $0$ (multiplicity two), where $\beta:=\sqrt{1+\sqrt{1+4\alpha^2}}$. The zeros of $\T_n'(z)$ are $0$ and $\pm1/\sqrt{2}$. The inverse image $\Tn$ is connected and consists of the interval $[-1,1]$ and two analytic Jordan arcs moving from $\frac{1}{2}\beta+\ii\alpha/\beta$ and $-\frac{1}{2}\beta+\ii\alpha/\beta$ to $\frac{1}{2}\beta-\ii\alpha/\beta$ and $-\frac{1}{2}\beta-\ii\alpha/\beta$ crossing the interval $[-1,1]$ at $1/\sqrt{2}$ and $-1/\sqrt{2}$, respectively. By Theorem\,\ref{T-InvPolImage}, the set $S=\Tn$ is the P\'olya-Chebotarev continuum for the $\nu=6$ points
\[
\{c_1,c_2,c_3,c_4,c_5,c_6\}=\{\pm1,\pm\tfrac{1}{2}\beta\pm\ii\alpha/\beta\}
\]
with the bifurcation points $d_{1,2}=\pm1/\sqrt{2}$. In Fig.\,\ref{Fig_ChebotarevSet-4}, we have plotted the inverse image $\Tn$ of the polynomial $\T_n(z)$ defined in \eqref{T4} for $\alpha=2$. For $\alpha=0$, we get $\T_n(z)=8z^4-8z^2+1$, i.e.\ the Chebyshev polynomial of the first kind of degree four.
\begin{figure}[ht]
\begin{center}
\includegraphics[scale=0.8]{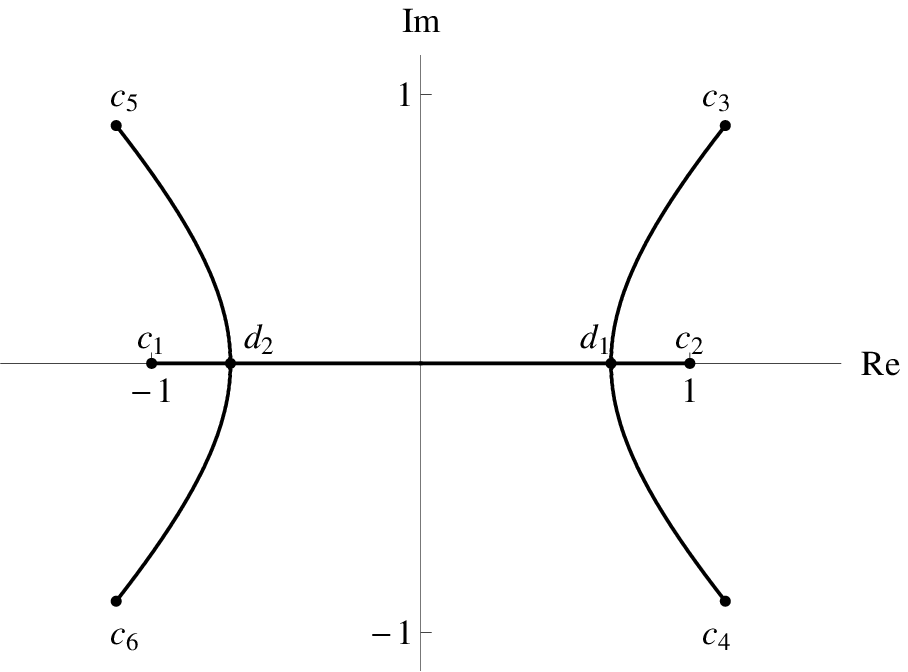}
\caption{\label{Fig_ChebotarevSet-4} The inverse image $\Tn$ of the polynomial $\T_n(z)$ defined in \eqref{T4} for $\alpha=2$}
\end{center}
\end{figure}
\end{enumerate}
\end{example}

\section{Construction of Polynomials with Connected Inverse Image}


Let us begin with a statement on a polynomial system of equations, which is valid for the zeros of $\T_n^2-1$, where $\T_n\in\PP_n$. For the proof, see \cite[Lemma\,2.1]{PehSch-1999}.


\begin{lemma}[\cite{PehSch-1999}]\label{L-PolEq}
\begin{enumerate}
\item Let $\T_n(z)=\tau{z}^n+\ldots\in\PP_n$ and let $z_1^+,\ldots,z_n^+$ and $z_1^-,\ldots,z_n^-$ be the zeros of $\T_n(z)-1$ and $\T_n(z)+1$, respectively. Then the following polynomial system of equations hold:
\begin{equation}\label{PolEq}
\sum_{j=1}^{n}(z_j^{+})^k-\sum_{j=1}^{n}(z_j^{-})^k=0, \qquad k=1,2,\ldots,n-1.
\end{equation}
\item Suppose that $z_1^{+},\ldots,z_n^{+},z_1^{-},\ldots,z_n^{-}\in\C$ with $z_i^{+}\neq{z}_j^-$, $i,j\in\{1,\ldots,n\}$, satisfy the polynomial system \eqref{PolEq}. Then there exists a polynomial $\T_n(z)=\tau{z}^n+\ldots\in\PP_n$ such that
\begin{equation}\label{Tn-1}
\T_n(z)-1=\tau\prod_{j=1}^{n}(z-z_j^{+}) \quad\text{and}\quad \T_n(z)+1=\tau\prod_{j=1}^{n}(z-z_j^{-})
\end{equation}
holds.
\end{enumerate}
In both cases, the leading coefficient $\tau$ of $\T_n$ is given by
\begin{equation}\label{tau-1}
\tau=-2\prod_{j=1}^{n}(z_1^{-}-z_j^{+})^{-1}.
\end{equation}
\end{lemma}


Next, let us give some further properties of the inverse image $\Tn$ of a polynomial $\T_n$, see \cite[Lemma\,1]{Sch-2012a}.


\begin{lemma}[\cite{Sch-2012a}]\label{L-InverseImage}
Let $\T_n\in\PP_n$.
\begin{enumerate}
\item The inverse image $\T_n^{-1}([-1,1])$ consists of $n$ analytic Jordan arcs, denoted by\\ $C_1,C_2,\dots,C_n$, where the $2n$ zeros of $\T_n^2-1$ are the endpoints of the $n$ arcs. If $z_0\in\C$ is a zero of $\T_n^2-1$ of multiplicity $\kappa$ then exactly $\kappa$ analytic Jordan arcs $C_{i_1},C_{i_2},\ldots,C_{i_\kappa}$ of $\T_n^{-1}([-1,1])$, $1\leq{i}_1<i_2<\ldots<i_\kappa\leq{n}$, have $z_0$ as common endpoint. These $\kappa$ Jordan arcs are cutting each other at successive angles of $2\pi/\kappa$. If $z_0\in\C$ is a double zero of $\T_n^2-1$ then the two analytic Jordan arcs with the same endpoint $z_0$ can be conjoined into one analytic Jordan arc.
\item The complement $\C\setminus\Tn$ is connected.
\end{enumerate}
\end{lemma}


With the help of Lemma\,\ref{L-PolEq}, we will construct polynomials $\T_n\in\PP_n$ with the property
\begin{equation}\label{z0}
\T_n'(z_0)=0 \quad\Rightarrow\quad \T_n(z_0)\in\{-1,1\}.
\end{equation}
Note that if $\T_n$ satisfies \eqref{z0} then, by Lemma\,\ref{L-Connectedness}, $\Tn$ is connected.


\begin{theorem}\label{T-Computation}
Suppose that the $2\nu-2$ pairwise distinct points $c_1,\ldots,c_{\nu}\in\C$ and\\ $d_1,\ldots,d_{\nu-2}\in\C$ satisfy the system
\begin{equation}\label{PolSystem-1}
\sum_{j=1}^{\nu}\alpha_jc_j^k+3\sum_{j=1}^{\nu-2}\gamma_jd_j^k+2\sum_{j=1}^{n-2\nu+3}\beta_jz_j^k=0, \qquad k=1,2,\ldots,n-1,
\end{equation}
where $z_1,\ldots,z_{n-2\nu+3}\in\C$ pairwise distinct, with $\alpha_j,\beta_j,\gamma_j\in\{-1,1\}$ and
\begin{equation}\label{PolSystem-2}
\sum_{j=1}^{\nu}\alpha_j+3\sum_{j=1}^{\nu-2}\gamma_j+2\sum_{j=1}^{n-2\nu+3}\beta_j=0.
\end{equation}
Then $S=\Tn$ is the P\'olya-Chebotarev continuum for $\{c_1,\ldots,c_{\nu}\}$ with bifurcation points $d_1,\ldots,d_{\nu-2}$, where $\T_n$ is given by
\begin{equation}\label{Tn-2}
\T_n(z)=1+\tau\prod_{\genfrac{}{}{0pt}{}{j=1}{\alpha_j=1}}^{\nu}(z-c_j)
\cdot\prod_{\genfrac{}{}{0pt}{}{j=1}{\gamma_j=1}}^{\nu-2}(z-d_j)^3
\cdot\prod_{\genfrac{}{}{0pt}{}{j=1}{\beta_j=1}}^{n-2\nu+3}(z-z_j)^2
\end{equation}
and
\begin{equation}\label{tau-2}
\tau=-2\Bigl(\prod_{\genfrac{}{}{0pt}{}{j=1}{\alpha_j=1}}^{\nu}(z^{-}-c_j)\cdot\prod_{\genfrac{}{}{0pt}{}{j=1}{\gamma_j=1}}^{\nu-2}(z^{-}-d_j)^3
\cdot\prod_{\genfrac{}{}{0pt}{}{j=1}{\beta_j=1}}^{n-2\nu+3}(z^{-}-z_j)^2\Bigr)^{-1},
\end{equation}
where $z^{-}$ is either a point $c_j$ for which $\alpha_j=-1$ or a point $d_j$ for which $\gamma_j=-1$ or a point $z_j$ for which $\beta_j=-1$.\\
\indent Moreover, the P\'olya-Chebotarev continuum $S=\Tn$ consists of $2\nu-3$ analytic Jordan arcs, where
\begin{itemize}
\item each point $c_j$, $j=1,\ldots,\nu$, is an endpoint of one analytic Jordan arc;
\item each point $d_j$, $j=1,\ldots,\nu-2$, is an endpoint of three analytic Jordan arcs with an angle of $2\pi/3$ between two arcs.
\end{itemize}
\end{theorem}


\begin{proof}
Suppose that \eqref{PolSystem-1} with \eqref{PolSystem-2} holds. By Lemma\,\ref{L-PolEq}, there exists a polynomial $\T_n\in\PP_n$ such that
\begin{equation}\label{Tn-3}
\T_n^2(z)-1=\tau^2\prod_{j=1}^{\nu}(z-c_j)\cdot\prod_{j=1}^{\nu-2}(z-d_j)^3\cdot\prod_{j=1}^{n-2\nu+3}(z-z_j)^2.
\end{equation}
Since each zero of $\T_n^2(z)-1$ of multiplicity $\kappa$ is a zero of $2\T_n(z)\T_n'(z)$ of multiplicity $\kappa-1$ thus a zero of $\T_n'(z)$ of multiplicity $\kappa-1$, we get
\begin{equation}\label{Tn'}
\T_n'(z)=n\tau\prod_{j=1}^{\nu-2}(z-d_j)^2\cdot\prod_{j=1}^{n-2\nu+3}(z-z_j).
\end{equation}
Thus all zeros of $\T_n'(z)$ are zeros of $\T_n^2(z)-1$, i.e.\ lie in $\Tn$, hence, by Lemma\,\ref{L-Connectedness}, $\Tn$ is connected.\\
\indent The representation \eqref{Tn-2} with \eqref{tau-2} of the polynomial $\T_n$ follows from \eqref{Tn-1} with \eqref{tau-1}.\\
\indent By \eqref{Tn-3}, $\T_n^2(z)-1$ has a factorization of the form \eqref{THU} with (note that $c_1,\ldots,c_{\nu}$ and $d_1,\ldots,d_{\nu-2}$ are pairwise distinct by assumption)
\begin{equation}\label{H-1}
\HH_{2\ell}(z):=\prod_{j=1}^{\nu}(z-c_j)\cdot\prod_{j=1}^{\nu-2}(z-d_j),
\end{equation}
i.e.\ $\ell=\nu-1$, and
\begin{equation}\label{U-1}
\U_{n-\ell}(z):=\tau\prod_{j=1}^{\nu-2}(z-d_j)\cdot\prod_{j=1}^{n-2\nu+3}(z-z_j).
\end{equation}
By \eqref{R}, \eqref{Tn'} and \eqref{U-1}, the polynomial $\RR_{\ell-1}(z)$ in Lemma\,\ref{L-THU} is
\[
\RR_{\ell-1}(z)=\prod_{j=1}^{\nu-2}(z-d_j),
\]
thus, by \eqref{H-1},
\[
\frac{\RR_{\ell-1}(z)}{\sqrt{\HH_{2\ell}(z)}}=\frac{\sqrt{\prod_{j=1}^{\nu-2}(z-d_j)}}{\sqrt{\prod_{j=1}^{\nu}(z-c_j)}},
\]
and we continue as in the proof of Theorem\,\ref{T-InvPolImage}.\\
\indent The last statement follows immediately by Lemma\,\ref{L-InverseImage} and the fact that the points $c_1,\ldots,c_{\nu}$ are simple zeros, the points $d_1,\ldots,d_{\nu-2}$ are triple zeros and the points\\ $z_1,\ldots,z_{n-2\nu+3}$ are double zeros of $\T_n^2(z)-1$, respectively.
\end{proof}


\begin{remark}
Consider the simplest case $\nu=3$. In view of \eqref{PolSystem-1} and \eqref{PolSystem-2}, there are 8 possible systems of equations (each with $k=1,\ldots,n-1$):
\begin{align*}
c_1^k+c_2^k+c_3^k+3d_1^k+2\sum_{j=1}^{(n-6)/2}(z_j^+)^k-2\sum_{j=1}^{n/2}(z_j^-)^k&=0\\
c_1^k-c_2^k-c_3^k+3d_1^k+2\sum_{j=1}^{(n-4)/2}(z_j^+)^k-2\sum_{j=1}^{(n-2)/2}(z_j^-)^k&=0\\
c_1^k-c_2^k+c_3^k-3d_1^k+2\sum_{j=1}^{(n-2)/2}(z_j^+)^k-2\sum_{j=1}^{(n-4)/2}(z_j^-)^k&=0\\
c_1^k+c_2^k-c_3^k-3d_1^k+2\sum_{j=1}^{(n-2)/2}(z_j^+)^k-2\sum_{j=1}^{(n-4)/2}(z_j^-)^k&=0\\
c_1^k+c_2^k+c_3^k-3d_1^k+2\sum_{j=1}^{(n-3)/2}(z_j^+)^k-2\sum_{j=1}^{(n-3)/2}(z_j^-)^k&=0\\
c_1^k+c_2^k-c_3^k+3d_1^k+2\sum_{j=1}^{(n-5)/2}(z_j^+)^k-2\sum_{j=1}^{(n-1)/2}(z_j^-)^k&=0\\
c_1^k-c_2^k+c_3^k+3d_1^k+2\sum_{j=1}^{(n-5)/2}(z_j^+)^k-2\sum_{j=1}^{(n-1)/2}(z_j^-)^k&=0\\
c_1^k-c_2^k-c_3^k-3d_1^k+2\sum_{j=1}^{(n-1)/2}(z_j^+)^k-2\sum_{j=1}^{(n-5)/2}(z_j^-)^k&=0
\end{align*}
Note that the first and second 4 systems are possible only for $n$ even and $n$ odd, respectively. In \cite{Sch-2013a}, the case $\nu=3$ is completely solved with the help of Jacobian elliptic and theta functions.
\end{remark}


\begin{remark}
If the assumptions of Theorem\,\ref{T-Computation} are satisfied and $\T_n$ is given by \eqref{Tn-2} with \eqref{tau-2}, by Theorem\,\ref{T-Computation} one can identify $\Tn$ as a \emph{simple ordinary graph} $G$ with no cycles, i.e.\ a \emph{tree}, with the $\nu$ vertices $c_1,\ldots,c_{\nu}$ of degree one, with the $\nu-2$ vertices $d_1,\ldots,d_{\nu-2}$ of degree $3$, and $2\nu-3$ edges (analytic Jordan arcs). The P\'olya-Chebotarev continuum in general can be interpreted as a simple ordinary graph (tree) with no vertices of degree $2$, see \cite[Definition\,1]{OrtegaPridhnani-2010}.
\end{remark}


\begin{example}
Consider the special case of $\nu=4$ points
\[
c_{1,2}=\alpha\pm\ii\beta, \qquad c_{3,4}=-\alpha\pm\ii\beta, \qquad \alpha,\beta>0,
\]
which are the vertices of a rectangle. If $\beta<\alpha$, it is known \cite{GrassmannRokne-1975} that the corresponding P\'olya-Chebotarev continuum is symmetric with respect to the real and imaginary axis and the corresponding bifurcation points $d_1,d_2$ lie on the real axis with $0<d_1<\alpha$ and $d_2=-d_1$, compare Fig.\,\ref{Fig_ChebotarevSet-5}. Without loss of generality, let us fix $\alpha=1$. For $n\in\{5,6,7,8,9\}$, below one can find the corresponding system of equations from \eqref{PolSystem-1} (note that $k=1,\ldots,n-1$) and the corresponding solution. For $n=9$, note that there are two possible systems. For both cases, we have plotted the corresponding inverse polynomial image in Fig.\,\ref{Fig_ChebotarevSet-5}.
\begin{figure}[ht]
\begin{center}
\includegraphics[scale=1.0]{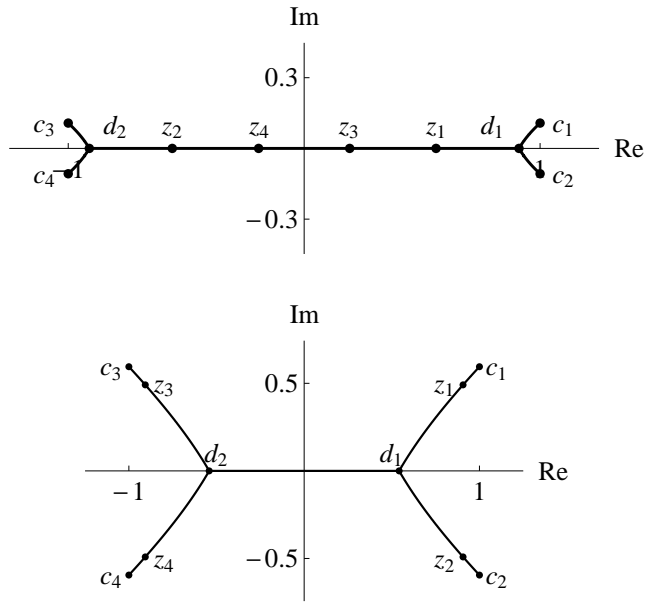}
\caption{\label{Fig_ChebotarevSet-5} Illustration of two P\'olya-Chebotarev continua for $\nu=4$ points which are the vertices of a rectangle}
\end{center}
\end{figure}
\begin{itemize}
\item $n=5$:
\[
\left\{
\begin{aligned}
&c_1^k+c_2^k-c_3^k-c_4^k-3d_1^k+3d_2^k=0\\
&\beta=\tfrac{\sqrt{5}}{3\sqrt{3}}, \quad d_{1,2}=\pm\tfrac{2}{3}
\end{aligned}
\right.
\]
\item $n=6$:
\[
\left\{
\begin{aligned}
&c_1^k+c_2^k+c_3^k+c_4^k-3d_1^k-3d_2^k+2z_1^k=0\\
&\beta=2-\sqrt{3}, \quad d_{1,2}=\pm2\sqrt{2-\sqrt{3}}, \quad z_1=0
\end{aligned}
\right.
\]
\item $n=7$:
\[
\left\{
\begin{aligned}
&c_1^k+c_2^k-c_3^k-c_4^k-3d_1^k+3d_2^k+2z_1^k-2z_2^k=0\\
&\beta\approx0.186748, \quad d_{1,2}\approx\pm0.848275, \quad z_{1,2}\approx\pm0.272412
\end{aligned}
\right.
\]
\item $n=8$:
\[
\left\{
\begin{aligned}
&c_1^k+c_2^k+c_3^k+c_4^k-3d_1^k-3d_2^k+2z_1^k-2z_2^k+2z_3^k=0\\
&\beta\approx0.138701, \quad d_{1,2}\approx\pm0.885782, \quad z_{1,3}\approx\pm0.442891, \quad z_2=0
\end{aligned}
\right.
\]
\item $n=9$:
\[
\left\{
\begin{aligned}
&c_1^k+c_2^k-c_3^k-c_4^k-3d_1^k+3d_2^k+2z_1^k-2z_2^k-2z_3^k+2z_4^k=0\\
&\beta\approx0.10749, \quad d_{1,2}\approx\pm0.910657, \quad z_{1,2}\approx\pm0.558978, \quad z_{3,4}\approx\pm0.192993
\end{aligned}
\right.
\]
\[
\left\{
\begin{aligned}
&c_1^k+c_2^k-c_3^k-c_4^k+3d_1^k-3d_2^k-2z_1^k-2z_2^k+2z_3^k+2z_4^k=0\\
&\beta\approx0.594803, \quad d_{1,2}\approx\pm0.541874, \quad z_{1,2}\approx0.906406\pm0.49118\,\ii,\\
&z_{3,4}\approx-0.906406\pm0.49118\,\ii
\end{aligned}
\right.
\]
\end{itemize}
\end{example}


Finally, let us state a conjecture on a certain density result:


\begin{conjecture}
Let $c_1,\ldots,c_{\nu}\in\C$ be $\nu$ distinct points. For each $\varepsilon>0$ there exists a polynomial $\T_n\in\PP_n$ such that $S=\Tn$ is the P\'olya-Chebotarev continuum for $\{\tilde{c}_1,\ldots,\tilde{c}_{\nu}\}$ with $|c_j-\tilde{c}_j|<\varepsilon$, $j=1,\ldots,\nu$.
\end{conjecture}


\bibliographystyle{amsplain}
\bibliography{Chebotarev-1}

\end{document}